\documentclass[11pt,reqno]{amsart}
\usepackage{amsfonts,amssymb,latexsym,amsthm,mathrsfs,cite}
\usepackage{graphicx}
\usepackage{hyperref}
\usepackage{xcolor}

\title{Random walks on rings and modules}
\author{Arvind Ayyer}
\address{Arvind Ayyer \\ Department of Mathematics\\
	Indian Institute of Science\\
	Bangalore 560012\\ India}
\email{arvind@math.iisc.ernet.in}
\author{Benjamin Steinberg}
\address{Benjamin Steinberg \\Department of Mathematics\\
    City College of New York\\
    Convent Avenue at 138th Street\\
    New York, New York 10031\\
    USA}
    \thanks{The first author was partially supported by the UGC Centre for Advanced Studies and by Department of Science and Technology grant EMR/2016/006624.
    The second author was supported by United States-Israel Binational Science Foundation \#2012080 and  by NSA MSP \#H98230-16-1-0047.}
\email{bsteinberg@ccny.cuny.edu}
\newcommand{\ann}{\mathrm{ann}}
\newcommand{\Aff}{\mathrm{Aff}}

\newcommand{\Stab}{\mathop{\mathrm {St}}\nolimits}

\newcommand{\Res}{\mathop{\mathrm{Res}}\nolimits}
\newcommand{\Ind}{\mathop{\mathrm{Ind}}\nolimits}

\newcommand{\JJ}{\mathrel{\mathscr J}} 
\newcommand{\RR}{\mathrel{\mathscr R}} 
\newcommand{\LL}{\mathrel{\mathscr L}} 

\newcommand{\inv}{^{-1}}
\newcommand{\p}{\varphi}
\newcommand{\pinv}{{\p \inv}}
\newcommand{\ov}[1]{\ensuremath{\overline {#1}}}

\newcommand{\wh}{\widehat}

\newcommand{\Hom}{\mathop{\mathrm{Hom}}\nolimits}

\newcommand{\chaptermarkb}[1]%
   {\markboth{{#1}}{}}
\newcommand{\sectionmarkb}[1]%
   {\markright{{#1}}{}}
\newcommand{\drawLinewithBG}[2]

\newtheorem{Thm}{Theorem}[section]
\newtheorem{Prop}[Thm]{Proposition}

{\theoremstyle{definition}
}
{\theoremstyle{remark}
\newtheorem{Rmk}[Thm]{Remark}}
\newtheorem{Cor}[Thm]{Corollary}
{\theoremstyle{remark}
}
{\theoremstyle{remark}
\newtheorem{example}[Thm]{Example}

{\theoremstyle{remark}
}
{\theoremstyle{remark}
}
{\theoremstyle{remark}
}

\date{\today}

\begin{document}

\begin{abstract}
We consider two natural models of random walks on a module $V$ over a finite commutative ring $R$ driven simultaneously by addition of random elements in $V$, and multiplication by random elements in $R$. In the coin-toss walk, either one of the two operations is performed depending on the flip of a coin. In the affine walk, random elements $a \in R,b \in V$ are sampled independently, and the current state $x$ is taken to $ax+b$.
For both models, we obtain the complete spectrum of the transition matrix from the representation theory of the monoid of all affine maps on $V$ under a suitable hypothesis on the measure on $V$ (the measure on $R$ can be arbitrary).
\end{abstract}

\maketitle

\section{Introduction}
\label{sec:intro}
Random walks driven simultaneously by addition and multiplication of the form $X_{n+1} = a_n X_n + b_n$, where $a_n,b_n$ are independent, have been considered in the past. Such chains seem to have been first studied  by Chung, Diaconis and Graham \cite{chung-diaconis-graham-1987} on the field $\mathbb{Z}/p\mathbb{Z}$, where $p$ is a prime. The motivation for their study came from the quest for efficient generation of quasirandom numbers, and those authors studied the mixing times of these chains for special choices of the distributions of $a_n$ and $b_n$. For different choices of these distributions, Hildebrand~\cite{hildebrand-1993, hildebrand-1996,hildebrand-2009} has also calculated the mixing times. Asci and
Hildebrand-McCollum generalized some of these results to the vector space $(\mathbb{Z}/p\mathbb{Z})^d$ \cite{asci-2001, hildebrand-mccollum-2008, asci-2009a,asci-2009b} with the restriction that $a_n$ is a deterministic matrix. In a slightly different direction, questions of convergence of affine random walks with real vectors are considered in~\cite{asci-2013}.

Another class of random walks driven by both these operations has the flavor
\begin{equation}
\label{coin-toss}
X_{n+1} =
\begin{cases}
X_n+b_n & \text{with probability $\alpha$}, \\
a_n X_n & \text{with probability $1-\alpha$},
\end{cases}
\end{equation}
where again $a_n$ and $b_n$ are chosen independently from some distribution.
Bate and Connor~\cite{bate-connor-2014} have determined the mixing times of such chains on $\mathbb{Z}/m\mathbb{Z}$ with $m$ odd and the probability $1-\alpha$ decreasing to zero as $m$ increases. Ayyer and Singla~\cite{AyyerSingla} have considered such chains on a finite commutative ring $R$ where the distribution for $b_n$ is uniform on $R$ and that of $a_n$ is arbitrary. They have determined the stationary distribution, the  spectrum of the transition matrix and the mixing time for finite chain rings.

We consider here the more abstract setting of random walks on finite left $R$-modules $V$ over finite commutative rings $R$.
Our methods will work both for the chains of the form $X_{n+1} = a_n X_n + b_n$, which we call the affine random walk, as well
as those of the chain described in \eqref{coin-toss}, which we call a coin-toss walk.
In both cases, $a_n$ and $b_n$ are independent elements of $R$ and $V$ respectively.  A word about notation: we shall use the phrase `random walk' to talk about such chains even though these chains are not reversible. Technically, random walks refer to reversible Markov chains, but this terminology has been established by various authors working on random walks on monoids (cf.~\cite{Rosenblatt,Brown1,Brown2,bjorner2, DiaconisBrown1,DiaconisICM,mobius1,mobius2,Saliolaeigen,ayyer_schilling_steinberg_thiery.2013,repbook}) and in the context of random affine mappings~\cite{chung-diaconis-graham-1987,bate-connor-2014} and so we shall continue to use it.
In this work, we focus on the spectral properties of these walks.
We plan to take up the study of probabilistic properties of our walks, such as the stationary distribution and the mixing time, in future. In particular, part of the motivation for our work is a systematic study of irreversible Markov chains. The intuition is that irreversible chains have faster mixing than reversible ones~\cite{DHN-2000, CH-2013, HM-2017, KJZ-2017}.

In this paper the term `ring' means unital ring.
Let $R$ be a finite commutative ring and $V$ a finite left $R$-module.  For both walks, we need probability distributions $P$ on $V$ and $Q$ on $R$.
As we shall see below, $Q$ can be arbitrary, but $P$ will have to satisfy a condition.
 The state space for both random walks will be the module $V$.

\subsection*{Coin-toss walk}
 At each step of the walk, we flip a coin which comes up heads with probability $\alpha$ and tails with probability $1-\alpha$.  If the result is heads, we move from $x\in V$ to $x+b$ with probability $P(b)$ and if the result is tails, we move from $x$ to $ax$ with probability $Q(a)$.

\subsection*{Affine walk}
 At each step of the walk, we independently choose $a\in R$ with probability $Q(a)$ and $b\in V$ with probability $P(b)$ and move to $ax+b$.  In other words, one step consists of first multiplying by an element of $R$ chosen randomly according to $Q$ and then adding an element of $V$ chosen according to $P$.

\bigskip
Both of these Markov chains can be viewed as random walks for the \emph{affine monoid} $\Aff(V)$ of $V$ where $\Aff(V)$ is the monoid of all mappings on $V$ of the form $x\mapsto ax+b$ with $a\in R$ and $b\in V$ with composition as the binary operation.  So the product of $ax+b$ and $cx+d$ is $acx+ad+b$.  Note that $\Aff(V)$ is the semidirect product of the multiplicative monoid $M(R)$ of $R$ with the additive group $V$, that is, $\Aff(V)= V\rtimes M(R)$, where $M(R)$ acts on $V$ via scalar multiplication.   We can view $P$ as a probability on $\Aff(V)$ supported on the translations $x\mapsto x+b$ and $Q$ as a probability on $\Aff(V)$ supported on the dilations $x\mapsto ax$.  The first model is then the random walk of $\Aff(V)$ on $V$ driven by the probability $\alpha P+(1-\alpha)Q$ and the second model is the random walk of $\Aff(V)$ on $V$ driven by the probability $PQ$ (where the product is convolution of measures).

To state our results, we need some definitions and notation. Recall that the set of invertible elements in a ring $R$ forms a group, known as the {\em group of units}.
Denote by $U(R)$ the group of units of $R$.  We make the convention that $R=\{0\}$ is a unital ring and that $U(R)=\{0\}$ is the trivial group.  Note that the only module over the zero ring is the zero module.
Two elements $r_1,r_2\in R$ are \emph{associates} if $r_1=ur_2$ with $u\in U(R)$.  It therefore seems natural to generalize this terminology to the module $V$ and so we say that $v_1,v_2\in V$ are \emph{associates} if $v_1=uv_2$ for some $u\in U(R)$. We shall in both models impose the additional assumption that associates are equally probable under $P$, that is, $P$ is constant on associates.  For example, this trivially holds for the uniform distribution. One can think of $V/U(R)$ as `projective space' and then we are asking that $P$ be a pullback of a measure on projective space.
We can also view $V/U(R)$ as the space of cyclic submodules of $V$, since, for a finite module $V$, $v_1,v_2\in V$ are associates if and only if $Rv_1=Rv_2$, that is, if and only if $v_1,v_2$ generate the same cyclic submodule.
Although this fact can be deduced from~\cite[Lemma~6.4]{Bass64},  we provide a proof of this for the reader's convenience that does not require as much background.

\begin{Prop}\label{p:cyclic.sub}
Let $V$ be a finite module over a finite  ring $R$.  Then, for $v,w\in V$, one has $Rv=Rw$ if and only if $U(R)v=U(R)w$ where $U(R)$ is the group of units of $R$.
\end{Prop}

We give a proof of Proposition~\ref{p:cyclic.sub} that relies only on the Krull-Schmidt theorem, following the second author's MathOverflow answer~\cite{mo211739}.

\begin{proof}[Proof of Proposition~\ref{p:cyclic.sub}]
Clearly, if $U(R)v=U(R)w$, then $Rv=Rw$. Turning to the converse, let $r,s\in R$
with $rv=w$ and $sw=v$. Since $R$ is finite, there exists $n>0$ such that $f=(rs)^n$ and $e=(sr)^n$ are idempotent ($n=|R|!$ will do). Note that $ev=v$ and $fw=w$. Let $r'=fre$ and $s'=esf$.  Notice that $r'v=w$. Moreover, $s'r'=esfre=es(rs)^nre=esr(sr)^ne=(sr)^{3n+1}=(sr)^{n+1}$ and so $Rs'r'=Re$.
Therefore, as $s'\in Rf$, right multiplication by $r'$ gives a surjective $R$-module homomorphism $Rf\to Re$.  Similarly right multiplication by $s'$ gives a surjective $R$-module homomorphism $Re\to Rf$. By finiteness of $R$ we conclude that both these homomorphisms are isomorphisms. It follows from the Krull-Schmidt theorem~\cite[Theorem~1.4.6]{benson} and the isomorphisms $Re\cong Rf$ and $Re\oplus R(1-e)\cong R\cong Rf\oplus R(1-f)$ that $R(1-f)\cong R(1-e)$.
Such an isomorphism $R(1-f)\to R(1-e)$ is given via right multiplication by an element $x\in (1-f)R(1-e)$.

Consider $u=r'+x$. Then $u$ is a unit since right multiplication by $u$ gives an isomorphism from $R=Rf\oplus R(1-f)$ to $R=Re\oplus R(1-e)$ (as it is the direct sum of the two isomorphisms $Rf\to Re$ and $R(1-f)\to R(1-e)$) and in a finite ring an element with a one-sided inverse is invertible. Also $uv=(r'+x)v=(r'+x)ev=r'ev=r'v=w$ because $x\in (1-f)R(1-e)$ implies $xe=0$. Thus $w\in U(R)v$.   This completes the proof.
\end{proof}

Notice that Proposition~\ref{p:cyclic.sub} implies that the natural map $U(R)\to U(R/I)$ is surjective for any ideal $I$ as the generators of the cyclic module $R/I$ are the units of $R/I$.

Denote by $\wh{A}$ the group of characters of an abelian group $A$, i.e., $\wh{A}=\Hom(A,U(\mathbb C))$. We shall denote by $\mathbf 1_A$ the \emph{trivial character} of $A$ mapping all of $A$ to $1$.  More generally, $\mathbf 1_G$ will denote the \emph{trivial representation} of any (not necessarily abelian) group $G$.

We write $\wh{V}$ for the character group of the additive group $(V,+)$ and $\widehat{U(R)}$ for the group of characters of the multiplicative group $U(R)$.  Since the correspondence $A\mapsto \wh A$ is contravariantly functorial and $R$ is commutative, the action of $R$ on $V$ by multiplication induces an action of $R$ on $\wh{V}$ by endomorphisms.  More precisely, if $\chi\in \wh{V}$ and $r\in R$, then $r\chi\colon V\to U(\mathbb C)$ is given by $(r\chi)(v) = \chi(rv)$ for $v\in V$.  In fact, this action turns $\wh{V}$ into a (finite) $R$-module since $(r+r')\chi(v)=\chi((r+r')v)=\chi(rv+r'v)=\chi(rv)\chi(r'v)=r\chi(v)\cdot r'\chi(v)$. Note that the  abelian group structure of the module $\wh V$ is being written multiplicatively because the operation is pointwise multiplication.   We call $\wh{V}$ the \emph{dual module} of $V$. Of course, $|\wh V|=V$ and $\wh{\wh V}$ is naturally isomorphic to $V$.  See~\cite{JWood} for more on the Pontryagin dual of a finite module.

Finally, recall that the {\em transition matrix} of a Markov chain on a finite state space $\Omega = \{\omega_1,\dots,\omega_n\}$ (in some ordering)
is the $n \times n$ matrix whose $(i,j)$-entry is given by the one-step probability of making a transition from $\omega_i$ to $\omega_j$.
Since the rows of the transition matrix sum to 1, it is said to be {\em row-stochastic}.
The next proposition gives sufficient conditions for the walks we are considering to be irreducible and aperiodic.  They are by no means necessary.

\begin{Prop}
\label{p:irred}
Let $P$ be a probability distribution on $V$ and $Q$ a probability distribution on $R$.
\begin{enumerate}
\item In the coin-toss walk with heads probability $0<\alpha<1$, if the support of $P$ generates the additive group of $V$, the walk is irreducible.  If, in addition, the monoid generated by the  support of $Q$ contains $0$, then the coin-toss walk is aperiodic.
\item If the support of $P$ generates the additive group of $V$ and the support of $Q$ contains $1$, then the affine walk is irreducible.  If, moreover, the submonoid generated by the support of $Q$ contains $0$, then the walk is aperiodic.
\end{enumerate}
\end{Prop}
\begin{proof}
For the first item, one can get from $v_1$ to $v_2$ with non-zero probability because the translation  $x\mapsto x+v_2-v_1$ is in the support of some convolution power of $\lambda=\alpha P+(1-\alpha)Q$ by our assumption on $P$.  If, in addition, $0$ is in the submonoid generated by the support of $Q$, then some convolution power of $\lambda$ contains a constant map in its support and hence, by irreducibility, there is a convolution power of $\lambda$ that contains all the constant maps in its support, cf.~\cite[Proposition~2.5]{ayyer_schilling_steinberg_thiery.2013}. The corresponding power of the transition matrix will be strictly positive. The argument for the second item is nearly identical, the assumption that $1$ is in the support of $Q$ being required to guarantee we can get any translation in the support of some convolution power of $PQ$.
\end{proof}

Since all the entries in the transition matrix are non-negative and at most $1$, all eigenvalues will have absolute value bounded above by $1$.

We can now state the main result of this article.

\begin{Thm}\label{t:main2}
Let $R$ be a finite commutative ring, $V$ a finite $R$-module, $P$ a probability on $V$ that is constant on associates and $Q$ a probability on $R$.
Then the eigenvalues for the transition matrices of both the coin-toss walk and the affine walk on $V$ are indexed by pairs $(W,\rho)$ where:
\begin{enumerate}
\item $W=R\chi$ is a cyclic $R$-submodule of $\wh{V}$;
\item and $\rho\in \wh{U(R/\ann(W))}$.
\end{enumerate}
The corresponding eigenvalue for the coin-toss walk is $\alpha\wh P(\chi)+(1-\alpha)\wh Q(\rho)$ and for the affine walk is $\wh P(\chi)\wh Q(\rho)$ where
\begin{align*}
\wh P(\chi) &= \sum_{b\in V}P(b)\chi(b)\\
\wh Q(\rho) &= \sum_{a\in U(R)+\ann(W)} Q(a)\rho(a+\ann(W))
\end{align*}
In both cases, the eigenvalue occurs with multiplicity one.
\end{Thm}

Note that in the above theorem statement, the value $P(\chi)$ depends only on the cyclic submodule generated by $\chi$ and not on $\chi$ itself.
It is an immediate corollary that, for generic choices of the parameters, the transition matrix is diagonalizable.

 The proof of Theorem~\ref{t:main2} is based on a careful analysis of the representation theory of the affine monoid $\Aff(V)$.  The use of monoid representation theory (outside of groups) to analyze Markov chains began with the work of Bidigare, Hanlon and Rockmore~\cite{BHR}, followed by work of Brown and Diaconis~\cite{DiaconisBrown1,Brown1,Brown2} and then others~\cite{bjorner1,bjorner2,DiaconisAthan,GrahamChung,mobius1,mobius2,ayyer_klee_schilling.2013,AyyerKleeSchilling,ayyer_schilling_steinberg_thiery.sandpile.2013,ayyer_schilling_steinberg_thiery.2013,Saliolaeigen,repbook}.  An introduction to these methods can be found in~\cite[Chapter~14]{repbook}.  All the papers cited above exploit the feature that the monoids in question only have one-dimensional irreducible representations over the field of complex numbers (or, equivalently, are faithfully representable by upper triangular matrices over the complex numbers~\cite{AMSV,repbook}).
 One novel element in this work is that the monoid in question, the affine monoid, has irreducible representations of higher dimensions.  This is the first article, to the best of our knowledge, to use the representation theory of a monoid that is neither a group, nor faithfully representable by upper triangular matrices, to analyze Markov chains.

We give an example to demonstrate Theorem~\ref{t:main2}.

\begin{example}
Let $R$ be the field $\mathbb{Z}/2\mathbb{Z}$ and $V$ the vector space $R^2$. Order the elements of $V$ as
$((0, 0), (0, 1), (1, 0), (1, 1))$.
The probability distribution on $R$ is $Q = (q_0,q_1)$ and that on $V$ is
$P = (p_{i,j})_{0 \leq i,j \leq 1}$. Being constant on associates forces no condition on $P$ for this simple example.
So the transition matrix of the affine walk is
{\small \[
\left(
\begin{array}{cccc}
 p_{0,0}  & p_{0,1}  & p_{1,0}  & p_{1,1}  \\
 q_0 p_{0,0}+q_1 p_{0,1} & q_1 p_{0,0}+q_0 p_{0,1} & q_0 p_{1,0}+q_1 p_{1,1} & q_1 p_{1,0}+q_0 p_{1,1} \\
 q_0 p_{0,0}+q_1 p_{1,0} & q_0 p_{0,1}+q_1 p_{1,1} & q_1 p_{0,0}+q_0 p_{1,0} & q_1 p_{0,1}+q_0 p_{1,1} \\
 q_0 p_{0,0}+q_1 p_{1,1} & q_0 p_{0,1}+q_1 p_{1,0} & q_1 p_{0,1}+q_0 p_{1,0} & q_1 p_{0,0}+q_0 p_{1,1} \\
\end{array}
\right),
\]}
and its eigenvalues are given by
\begin{align*}
(1-\alpha)q_1 & \left(p_{0,0}+ p_{0,1}-p_{1,0}-p_{1,1}\right),
q_1 \left(p_{0,0}-p_{0,1}+p_{1,0}-p_{1,1}\right), \\
& q_1 \left(p_{0,0}-p_{0,1}-p_{1,0}+p_{1,1}\right)
\text{ and } 1.
\end{align*}
The transition matrix of the coin-toss walk is
{\footnotesize\[
\left(
\begin{array}{cccc}
 (1-\alpha) +\alpha  p_{0,0} & \alpha  p_{0,1} & \alpha  p_{1,0} & \alpha  p_{1,1} \\
 (1-\alpha) q_0 +\alpha  p_{0,1} & (1-\alpha) q_1+\alpha  p_{0,0} & \alpha  p_{1,1} & \alpha  p_{1,0} \\
 (1-\alpha) q_0 +\alpha  p_{1,0} & \alpha  p_{1,1} & (1-\alpha) q_1+\alpha  p_{0,0} & \alpha  p_{0,1} \\
 (1-\alpha) q_0 +\alpha  p_{1,1} & \alpha  p_{1,0} & \alpha  p_{0,1} & (1-\alpha) q_1+\alpha  p_{0,0} \\
\end{array}
\right),
\]}
and its eigenvalues are $1$,
\begin{align*}
&\alpha \left(p_{0,0}+p_{0,1}-p_{1,0}-p_{1,1}\right)+(1-\alpha) q_1, \\
& \alpha  \left(p_{0,0}-p_{0,1}+p_{1,0}-p_{1,1}\right)+(1-\alpha) q_1, \\
\text{and } & \alpha  \left(p_{0,0}-p_{0,1}-p_{1,0}+p_{1,1}\right)+(1-\alpha) q_1.
\end{align*}
\end{example}

Theorem~\ref{t:main2} has a nice reformulation for Frobenius rings when $V$ is the ring $R$ itself.  This includes rings of the form $\mathbb Z/n\mathbb Z$.  Recall that an Artinian ring $R$ (with Jacobson radical $J(R)$) is \emph{Frobenius}  if $R/J(R)$ is isomorphic to the socle of $R$ as a left module and as a right module. (Recall that the socle of a module is its largest semisimple submodule.)  It is shown in~\cite{JWood} that a finite ring $R$ is Frobenius if and only if $R\cong \wh{R}$ as a left $R$-module.  A character $\chi\in \wh{R}$ such that $r\mapsto r\chi$ is an $R$-module isomorphism is called a \emph{generating character}~\cite{JWood}.  For example, $R=\mathbb Z/n\mathbb Z$ is Frobenius and a generating character is given by $\chi(m) = e^{2\pi im/n}$ for $m\in\mathbb Z/n\mathbb Z$.

\begin{Thm}\label{t:main3}
Let $R$ be a finite commutative Frobenius ring,  $P$ a probability on $R$ that is constant on associates and $Q$ a probability on $R$.  Let $\chi\in \wh R$ be a generating character.  Then the eigenvalues for the transition matrices of both the coin-toss walk and the affine walk on $R$ are indexed by pairs $(W,\rho)$ where:
\begin{enumerate}
\item $W=Rb$ is a principal ideal;
\item and $\rho\in \wh{U(R/\ann(b))}$.
\end{enumerate}
The corresponding eigenvalue for the coin-toss walk is
\[
\alpha\cdot\sum_{r\in R}P(r)\chi(br)+(1-\alpha)\cdot \!\!\!\!\!\!
\sum_{r\in U(R)+\ann(b)}Q(r)\rho(r+\ann(b))
\] and for the affine walk is
\[
\left(\sum_{r\in R}P(r)\chi(br) \right) \cdot \left(\sum_{r\in U(R)+\ann(b)}Q(r)\rho(r+\ann(b)) \right).
\]
In both cases, the eigenvalue occurs with multiplicity one.
\end{Thm}

\begin{example}
Let $R$ be the ring $\mathbb{Z}/4\mathbb{Z}$, which we order as $(0,1,2,3)$, and let $V=R$. Since $1$ and $3$ are associates, we set $p_3 = p_1$.
The transition graphs of the coin-toss walk and the  affine walk are shown in Figure~\ref{fig:ex-Z4}. The eigenvalues of their respective transition matrices  are
\begin{align*}
&\alpha\left(p_0-p_2\right) +(1-\alpha)\left(q_1-q_3\right),
\alpha\left(p_0-p_2\right) +(1-\alpha)\left(q_1+q_3\right), \\
&\alpha\left(p_0-2 p_1+p_2\right) +(1-\alpha)\left(q_1+q_3\right)
\text{ and } 1,
\end{align*}
and
\[
\left(p_0-p_2\right) \left(q_1-q_3\right),\left(p_0-p_2\right) \left(q_1+q_3\right),\left(p_0-2 p_1+p_2\right) \left(q_1+q_3\right)
\text{ and } 1.
\]

\begin{figure}[htbp!]
\begin{center}
\begin{tabular}{c c}
\includegraphics[scale=0.28]{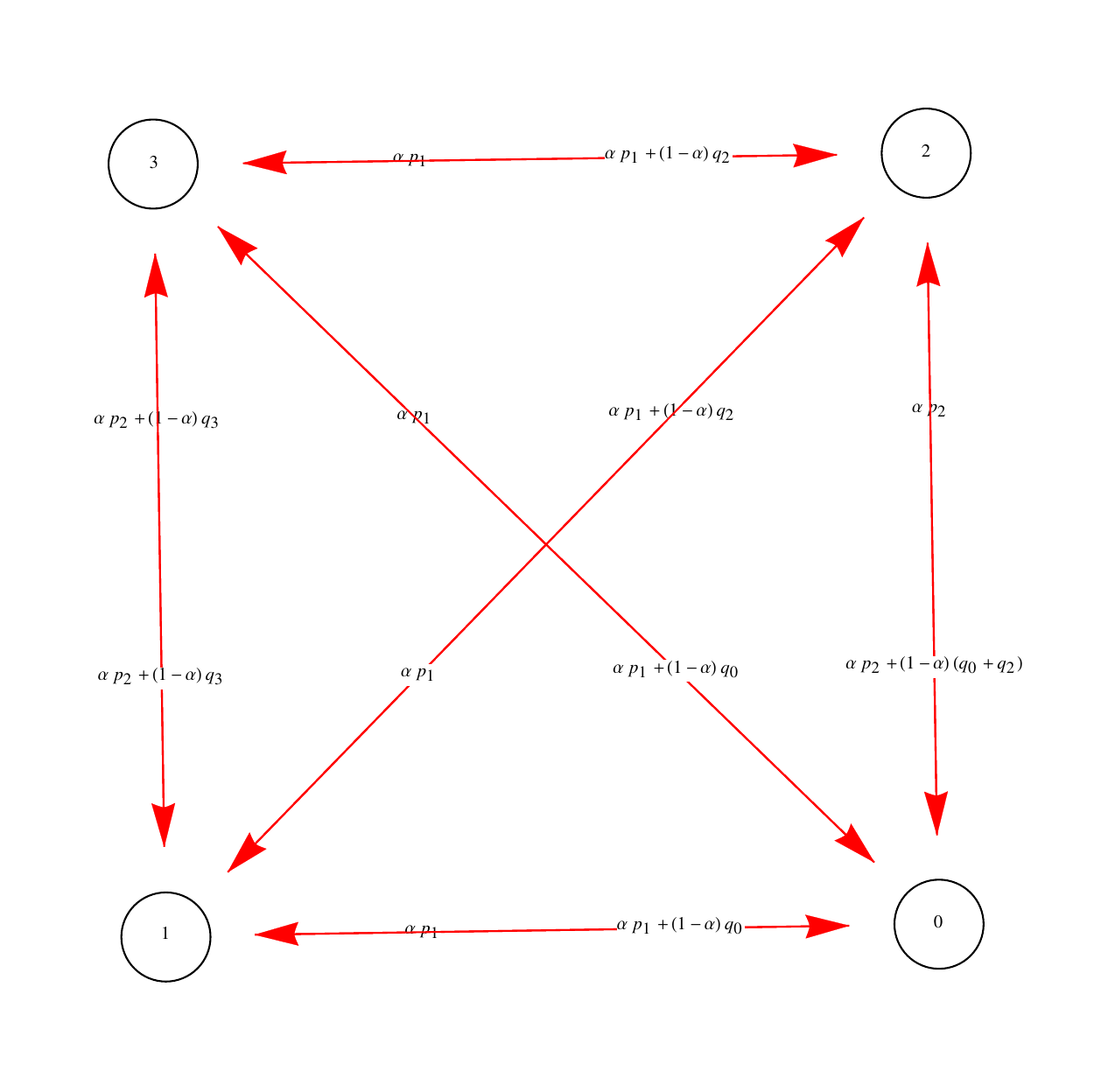}
&
\includegraphics[scale=0.28]{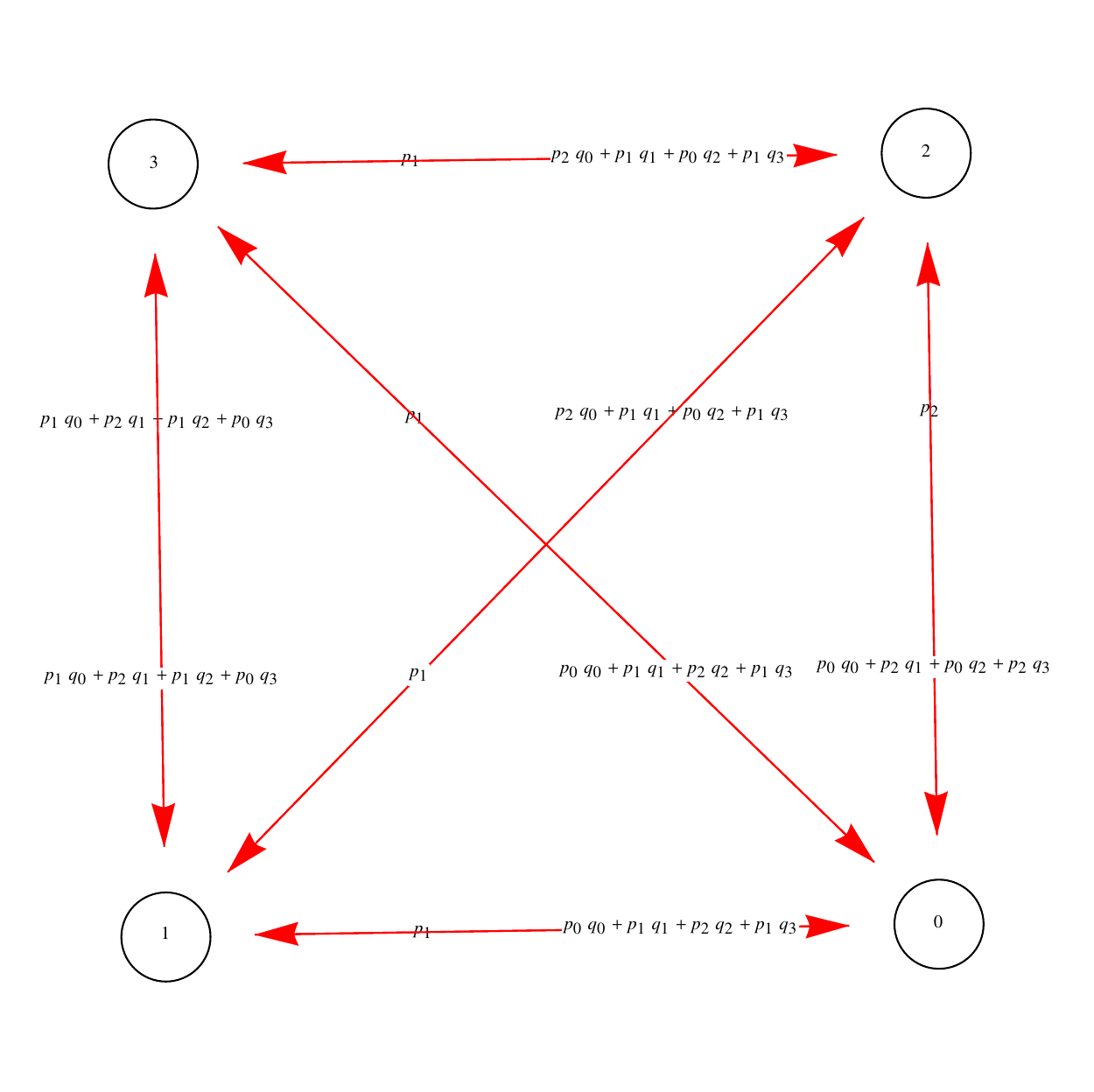}
\end{tabular}
\caption{The transition graphs of the coin-toss and affine walk on $\mathbb{Z}/4\mathbb{Z}$.}
\label{fig:ex-Z4}
\end{center}
\end{figure}
\end{example}

The plan of the rest of the paper is as follows. In Section~\ref{sec:prelims}, we review first the representation theory of groups followed by that of monoids, emphasizing the parts that are relevant to this study. In Section~\ref{sec:affine}, we study the structure of the affine monoid and use the results in the previous section to understand the representation theory of this monoid. Finally, we prove Theorem~\ref{t:main2} and examine various special cases of this general result in Section~\ref{sec:eigen}.

\section{Preliminaries on group and monoid representation theory}
\label{sec:prelims}
The book~\cite{repbook} serves as a basic reference for those aspects of the representation theory of monoids that we shall need.
If $M$ is a finite monoid, then $\mathbb CM$ denotes the \emph{monoid algebra} of $M$. It consists of all formal linear combinations of elements of $M$ with the obvious addition operation and with product \[\sum_{m\in M}c_mm\cdot \sum_{m\in M}d_mm=\sum_{m,n\in M}c_md_nmn.\]  Finite dimensional $\mathbb CM$-modules (which are the only kind we consider) correspond to finite dimensional matrix representations of $M$ over $\mathbb C$.  A probability $P$ on $M$ can be identified with the element $\sum_{m\in M}P(m)m\in \mathbb CM$ and the product of two probabilities in $\mathbb CM$ corresponds to their convolution.

If $V$ is a $\mathbb CM$-module, the \emph{character} of $V$ is the mapping $\chi_V\colon M\to \mathbb C$ given by sending $m\in M$ to the trace of the operator on $V$ given by $v\mapsto mv$.  It is not in general true that a module is determined by its character but semisimple modules are, cf.~\cite{RhodesZalc}.  The character of a simple module is called an \emph{irreducible character}.  The irreducible characters of a monoid form a linearly independent set of mappings~\cite[Theorem~7.7]{repbook}.

A \emph{composition series} for a $\mathbb CM$-module $V$ is a series of submodules
\begin{equation}\label{eq.comp.series}
V=V_0\supsetneq V_1\supsetneq \cdots \supsetneq V_n=0
\end{equation}
such that the \emph{composition factor $V_i/V_{i+1}$} is simple for $i=0,\ldots, n-1$.  The Jordan-H\"older theorem~\cite[Theorem~1.1.4]{benson}
guarantees that the length of any two composition series for $V$ is the same and, moreover, that if $S$ is a simple $\mathbb CM$-module, then the number $[V:S]$ of composition factors isomorphic to $S$ is the same for any two composition series.  The isomorphism class of a module $V$ shall be written $[V]$.

\subsection{Group representation theory}
The reader is referred to~\cite{serrerep} for the basics of group representation theory.  If $G$ is a finite group, then $\mathbb CG$ is called the \emph{group algebra} of $G$.  It is a semisimple algebra and hence every finite dimensional $\mathbb CG$-module $V$ is a direct sum of simple $\mathbb CG$-modules, which are in fact its composition factors (with multiplicity).  If $V$ is a $\mathbb CG$-module and $S$ is a simple $\mathbb CG$-module, then
\begin{equation}\label{mult.form}
 \dim \Hom_{\mathbb CG}(S,V)=\dim \Hom_{\mathbb CG}(V,S)=[V:S].
\end{equation}

The irreducible characters of a finite group $G$ form an orthonormal set for the inner product on $\mathbb C^G$ given by
\[\langle f,g\rangle = \frac{1}{|G|}\sum_{x\in G}f(x)\ov{g(x)}.\] This is called the \emph{first orthogonality relations}. If $\theta$ is the character of a $\mathbb CG$-module $V$ and $\chi$ is the character of a simple $\mathbb CG$-module $S$, then $\langle \theta,\chi\rangle=[V:S]$.

An important consequence of Schur's lemma is that if $a$ belongs to the center of $\mathbb CG$ and $V$ is a simple $\mathbb CG$-module, then $a$ acts on $V$ via multiplication by a scalar.

If $H\leq G$ is a subgroup and $V$ is a $\mathbb CH$-module, then $\Ind_H^G V=\mathbb CG\otimes_{\mathbb CH} V$ is a $\mathbb CG$-module called an \emph{induced module}.  Note that \[\dim \Ind_H^G V=[G:H]\cdot \dim V.\]  Also, $\Ind _H^G \mathbf 1_H$ is isomorphic to the permutation module $\mathbb C[G/H]$.  If $W$ is a $\mathbb CG$-module, then $\Res^G_H W$ is the $\mathbb CH$-module obtained by restricting scalars.

\begin{Thm}[Frobenius reciprocity]\label{t:frob.rec}
Let $H\leq G$ be a subgroup and let $V$ be a $\mathbb CH$-module and $W$ a $\mathbb CG$-module.  Then the isomorphism
\[\Hom_{\mathbb CG}(\Ind_H^G V,W)\cong \Hom_{\mathbb CH}(V,\Res^G_H W)\]
holds.
\end{Thm}

The Mackey decomposition theorem describes how an induced representation from one subgroup restricts to another.  Let $K\leq G$ be a subgroup and let $V$ be a $\mathbb CK$-module.  If $g\in G$, then $V^g$ is the $\mathbb C[gKg^{-1}]$-module with underlying vector space $V$ and action given by $xv=g^{-1}xgv$ for $x\in gKg^{-1}$.   With this notation, the Mackey decomposition theorem says the following.

\begin{Thm}[Mackey decomposition]\label{t:mackey.dec}
Let $G$ be a group and let $H,K$ be subgroups of $G$.  Let $T$ be a complete set of representatives of the double cosets $H\backslash G/K$.  If $V$ is a $\mathbb CK$-module, then the decomposition
\[\Res^G_H \Ind_K^G V\cong \bigoplus_{t\in T} \Ind_{H\cap tKt^{-1}}^H \Res ^{tKt^{-1}}_{H\cap tKt^{-1}}V^t\]
holds.
\end{Thm}

If $\p\colon G\to K$ is a group homomorphism and $V$ is a $\mathbb CK$-module, then it is also a $\mathbb CG$-module, called the \emph{inflation} of $V$ along $\p$, via $gv=\p(g)v$ for $g\in G$ and $v\in V$.

\begin{Prop}\label{p:inflate.ind}
Let $\p\colon G\to K$ be a surjective homomorphism, $H\leq K$ and $V$ a $\mathbb CH$-module. Putting $H'=\pinv(H)$, the inflation of $\Ind_H^K V$ along $\p$ is isomorphic to $\Ind_{H'}^G V$ (where $V$ is a $\mathbb CH'$-module via inflation).
\end{Prop}
\begin{proof}
It is easily verified that there is a surjective homomorphism of $\mathbb CG$-modules $\psi\colon \mathbb CG\otimes_{\mathbb CH'} V\to \mathbb CK\otimes_{\mathbb CH} V$ given by $g\otimes v\mapsto \p(g)\otimes v$ on basic tensors.
Since $\dim \Ind_{H'}^G V= [G:H']\cdot \dim V=[K:H]\cdot \dim V=\dim \Ind_H^K V$, we conclude that $\psi$ is an isomorphism.
\end{proof}

\subsection{Monoid representation theory}

Just as for rings, the set of invertible elements (known as units) in a monoid forms a group.
The \emph{group of units} (invertible elements) of a monoid $M$ will be denoted $U(M)$ throughout.
An \emph{idempotent} of a monoid $M$ is an element $e\in M$ such that $e^2=e$.  Denote by $E(M)$ the set of idempotents of $M$.  If $M$ is a monoid and $e\in E(M)$ is an idempotent, then $eMe$ is a monoid with identity $e$.
The group of units $U(eMe)$ of $eMe$ is called the \emph{maximal subgroup} of $M$ at $e$.

 We recall the definition of Green's relations~\cite{Green} on a monoid $M$.  We write $m\JJ n$ if $MmM=MnM$, $m\RR n$ if $mM=nM$ and $m\LL n$ if $Mm=Mn$.  We write $m\leq_{\JJ} n$ if $MmM\subseteq MnM$.  A $\JJ$-class is called \emph{regular} if it contains an idempotent. The regular $\mathscr J$-classes of  $M$ form a poset via the ordering $J\leq J'$ if $MJM\subseteq MJ'M$.  If $e$ and $f$ are $\JJ$-equivalent idempotents, then $eMe\cong fMf$ and hence $U(eMe)\cong U(fMf)$.  See~\cite[Corollary~1.2]{repbook}.

Let us first recall some basic facts about the representation theory of monoids.  If $e\in E(M)$ and $V$ is a $\mathbb CM$-module, then $eV$ is a $\mathbb CU(eMe)$-module.  We say that $e$ is an \emph{apex} for the simple $\mathbb CM$-module $S$ if $eS\neq 0$ and $mS=0$ for all $m\in eMe\setminus U(eMe)$.  This is equivalent to $mS=0$ for all $m$ such that $e\notin MmM$.  The fundamental theorem of Clifford-Munn-Ponizovskii theory says the following.  See~\cite{gmsrep} or ~\cite[Theorem~5.5]{repbook} for details.

\begin{Thm}\label{t:CMP}
Let $M$ be a finite monoid and  $e_1,\ldots, e_s$ form a complete set of idempotent representatives of the regular $\mathscr J$-classes of $M$.
Then the isomorphism classes of simple $\mathbb CM$-modules are parameterized by pairs $(e_i,[V])$ with $V$ a simple $\mathbb CU(e_iMe_i)$-module.  The corresponding simple module $V^\sharp$ is characterized up to isomorphism by the properties that $e_i$ is an apex for $V^\sharp$ and $e_iV^\sharp\cong V$ as a $\mathbb CU(e_iMe_i)$-module.
\end{Thm}

McAlister~\cite{McAlisterCharacter} gave a general method to compute the composition factors of a module from its character by inverting the character table of the monoid.  In practice, it can be quite unwieldy to implement this method.  A simpler method was given by the second author in~\cite{mobius2} in the case of a monoid whose idempotents form a submonoid. The reader is referred to~\cite{Stanley} for the M\"obius function of a poset.

\begin{Thm}\label{t:mult}
Let $M$ be a finite monoid whose idempotents form a submonoid. Fix an idempotent $e_J$ from each regular $\mathscr J$-class $J$.  Let $\chi$ be an irreducible character of $U(e_JMe_J)$ corresponding to the simple module $S$ and let $V$ be a finite dimensional $\mathbb CM$-module with character $\theta$.  Then
\[[V:S^\sharp]=\frac{1}{|U(e_JMe_J)|}\sum_{g\in U(e_JMe_J)} \chi(g)\sum_{J'\leq J}\ov{\theta(e_{J'}ge_{J'})}\mu(J',J)\] where $J'$ runs over regular $\mathscr J$-classes and $\mu$ is the M\"obius function of the poset of regular $\mathscr J$-classes of $M$.
\end{Thm}

\subsection{Monoid random walks}\label{ss.monoid.random}
Let $M$ be a finite monoid acting on the left of a finite set $\Omega$ and let $P$ be a probability on $M$.  Then the \emph{random walk} of $M$ on $\Omega$ driven by $P$ is the Markov chain with state space $\Omega$ and with transitions $x\mapsto mx$ with probability $P(m)$.

The vector space $\mathbb C\Omega$ is then a left $\mathbb CM$-module by extending the action of $M$ on the basis $\Omega$ linearly.   If we identify $P$ with the element  \[\sum_{m\in M}P(m)m\in \mathbb CM,\] then the transition matrix of the random walk is the transpose of the matrix of the operator $P$ acting on the vector space $\mathbb C\Omega$ with respect to the basis $\Omega$.  See~\cite[Chapter~14]{repbook} or~\cite{Brown1,Brown2,ayyer_schilling_steinberg_thiery.2013} for details.  Note that the tranpose arises here because we are using row stochastic matrices for the transition matrix but left actions for the random walk.

\section{The affine monoid}
\label{sec:affine}
As before, fix a finite commutative ring $R$ and a finite $R$-module $V$.
We continue to use $\Aff(V)$ to denote the affine monoid of $V$.  Observe that \[U(\Aff(V)) =\{ax+b\mid a\in U(R), b\in V\}\] and, in fact, $U(\Aff(V)) = V\rtimes U(R)$ is a semidirect product of abelian groups.

\subsection{The algebraic structure of the affine monoid}
If $e\in E(R)$, then $1-e\in E(R)$ and the internal direct sum \[R=Re\oplus R(1-e)\]  is a direct product decomposition as rings (note that $Rf$ is a unital ring with identity $f$ for any $f\in E(R)$).  All direct sum decompositions $R=R_1\oplus R_2$ into a direct product of unital rings arise in this way (take $e$ to be the identity of $R_1$). There is a corresponding direct sum decomposition $V=eV\oplus (1-e)V$ and note that $eV$ is annihilated by $R(1-e)$ and $(1-e)V$ is annihilated by $Re$.  Also, $eV$ is an $Re$-module.
Denote by \[\varphi_e\colon R=Re\oplus R(1-e)\to Re\] the projection.  It is a surjective homomorphism of unital rings given by $\varphi_e(r)=re$. Notice that $re=0$ if and only if $r(1-e)=r$ and so $Re\cong R/R(1-e)$.  We also have a surjective homomorphism of $R$-modules $\pi_e\colon V\to eV$ given by $\pi_e(v)=ev$, which has kernel $(1-e)V$.

\begin{Prop}\label{p:unit.onto}
Let $e\in E(R)$ and let $\p_e\colon R\to Re$ be the projection $\varphi_e(r)=re$.
\begin{enumerate}
\item For $a\in R$, $\p_e(a)\in U(Re)$ if and only if $Ra\supseteq Re$.
\item The restriction $\p_e\colon U(R)\to U(Re)$ is surjective.
\end{enumerate}
\end{Prop}
\begin{proof}
To prove the first item, let $Ra\supseteq Re$ and write $e=ya$.  Then $e=yeae=ye\varphi_e(a)$ and so $\varphi_e(a)\in U(Re)$.  Conversely, if $ae\in U(Re)$, then $e=uae$ with $u\in U(Re)$ and hence $e=uea\in Ra$.  Thus $Ra\supseteq Re$.

For the second item, let $u\in U(Re)$ with inverse $v\in U(Re)$, and so $uv=e$.  Let $u'=u+(1-e)$ and $v'=v+(1-e)$.  Then since $ue=u$, $ve=v$ and $e(1-e)=0$, we see that $u'v' = uv+u(1-e)+v(1-e)+(1-e) = uv+(1-e)=e+1-e=1$ and so $u'\in U(R)$.  Moreover, $\varphi_e(u')=u'e=u$.  This completes the proof.
\end{proof}

We now begin to study the affine monoid.  Let $\pi\colon \Aff(V)\to M(R)$ be the surjective homomorphism $\pi(ax+b)=a$.   A monoid is called a \emph{left regular band} if it satisfies the identity $xyx=xy$.  Left regular bands have played an important role in applications of monoids to Markov chain theory. See~\cite{BHR,DiaconisBrown1,Brown1,Brown2,DiaconisICM}.

\begin{Prop}\label{p:Aff.mon.idem}
Let $R$ be a commutative ring.
\begin{enumerate}
  \item $E(\Aff(V)) = \{ex+b\mid e\in E(R), eb=0\}$.
  \item $E(\Aff(V))$ is a submonoid of $\Aff(V)$ and a left regular band.
  \item If $ex+b,fx+c\in E(\Aff(V))$, then $ex+b\leq_{\JJ} fx+c$ if and only if $Re\subseteq Rf$. In particular, $ex+b\JJ fx+c$ if and only if $e=f$.
  \item If $e\in E(R)$, then $E(Re) = \{f\in E(R)\mid Rf\subseteq Re\}$.
\end{enumerate}
\end{Prop}
\begin{proof}
If $f(x)=ax+b$, then $f^2(x) = a^2x+ab+b$ and so $f(x)=f^2(x)$ if and only if $a^2=a$ and $ab=0$.  This proves the first item.

For the second item, suppose that $e,f\in E(R)$ and $eb=0=fc$. Let $g(x)=ex+b$ and $h(x)=fx+c$.  Then $gh(x)= efx+ec+b$, $ef\in E(R)$ and $ef(ec+b) = efc+feb=0$.  Thus $gh(x)$ is an idempotent by the first item and so the idempotents form a submonoid of $\Aff(R)$. Moreover, $ghg(x) = efex+efb+ec+b=efx+ec+b$ as $eb=0$.  Therefore, $ghg(x)=gh(x)$ and so $E(\Aff(V))$ is a left regular band.

Let us prove the third item.  Let $g(x)=ex+b,h(x)=fx+c\in E(\Aff(V))$.  If $g(x)\leq_{\JJ} h(x)$, then $e=\pi(g(x))\leq_{\JJ} \pi(h(x))=f$.  But $M(R)$ is a commutative monoid and so $Re\subseteq Rf$.  Conversely, if $Re\subseteq Rf$, then $ef=e$ and so $gh(x) = efx+ec+b=ex+efc+b=ex+b=g(x)$ because $fc=0$.  Thus $g(x)\leq_{\mathscr J} h(x)$. The final statement of the third item follows because $ex+b\JJ fx+c$ if and only if $Re=Rf$.  But if $Re=Rf$, then $e=ef=fe=f$.

The last item follows because an idempotent $f$ belongs to $Re$ if and only if $fe=f$ if and only if $Rf\subseteq Re$.
\end{proof}

Proposition~\ref{p:Aff.mon.idem} implies that the poset of regular $\mathscr J$-classes of $\Aff(V)$ is isomorphic to the poset of idempotent generated principal ideals of $R$.  The latter is, in fact, a lattice since $Re\cap Rf=Ref$ and $Re+Rf = R(e+f-ef)$ and one readily checks $e+f-ef$ is an idempotent.  In fact, the lattice of idempotent-generated principal ideals of $R$ is a boolean algebra where the complement of $Re$ is $R(1-e)$ and the atoms are the ideals generated by the primitive idempotents.

\begin{Cor}\label{c:poset.reg}
The poset of regular $\mathscr J$-classes of $\Aff(V)$ is isomorphic to the lattice of idempotent generated principal ideals of the ring $R$ via the mapping taking the $\mathscr J$-class of $ex+b$ to $Re$.
\end{Cor}

The elements of the form $g(x)=ex$ with $e\in E(R)$ form a commutative submonoid of $E(\Aff(V))$ and a transversal of the set of regular $\mathscr J$-classes. Moreover, $Re\subseteq Rf$ if and only if $e=ef=fe$.

\begin{Prop}\label{p:idem.local}
Let $e\in E(R)$ and let $g(x)=ex$ be the corresponding idempotent of $\Aff(V)$. Then $g(x)\Aff(V)g(x)=\Aff(eV)$, where $eV$ is viewed as an $Re$-module in the natural way, and hence the maximal subgroup of $\Aff(V)$ at $g(x)$ is $U(\Aff(eV))$.
\end{Prop}
\begin{proof}
If $h(x)=ax+b$, then $ghg(x) = eaex+eb = eax+eb$.  It follows that $ax+b\in g(x)\Aff(V)g(x)$ if and only if $ea=ae=a$ and $eb=b$, that is, $g(x)\Aff(V)g(x) = \Aff(eV)$, establishing the proposition.
\end{proof}

It will be useful that the projections $\p_e\colon R\to Re$  and $\pi_e\colon V\to eV$, for $e\in E(R)$,  combine to yield a homomorphism $\Aff(V)\to \Aff(eV)$.

\begin{Prop}\label{p:extend.proj}
Let $e\in E(R)$.  Then there is a surjective homomorphism $\Phi_e(\Aff(V))\to \Aff(eV)$ given by $\Phi_e(ax+b) = \p_e(a)x+\pi_e(b)=aex+eb$.  Moreover, the restriction $\Phi_e\colon U(\Aff(V))\to U(\Aff(eV))$ is surjective.
\end{Prop}
\begin{proof}
Let $g(x)=ax+b$ and $h(x)=cx+d$.  Then $gh(x) = acx+ad+b$ and so $\Phi_e(gh(x)) = acex+ead+eb$.  On the other hand, $\Phi_e(g(x))\Phi_e(h(x))$ is the composition of $aex+eb$ and $cex+ed$, which is $aecex+aeed+eb=acex+ead+eb$. Also $\Phi_e(1x) = ex$, which is the identity of $\Aff(eV)$.  Therefore, $\Phi_e$ is a homomorphism.  It is surjective because if $a\in Re$ and $b\in eV$, then $ae=a$ and $eb=b$, whence $\Phi_e(ax+b) = ax+b$.  The final statement follows immediately from Proposition~\ref{p:unit.onto}.
\end{proof}

Note that if $e\in E(R)$, then  $e\wh{V}$ is an $R$-submodule of $\wh{V}$, and hence invariant under its group of units, $U(R)$.  Also note that $e\wh{V}$ is an $Re$-module. So the unit group $U(Re)$ acts on $e\wh{V}$ by automorphisms. The stabilizer in $U(Re)$ of $\chi\in e\wh{V}$ under this action shall be denoted $\Stab_{U(Re)}(\chi)$.  Notice that $\Stab_{U(Re)}(\chi)=e+\ann(\chi)e$, where $\ann(\chi)$ is the ideal of elements of $R$ that annihilate $\chi$; indeed, $r\chi=\chi=e\chi$ if and only if $r-e\in \ann(\chi)e$, if and only if $r\in e+\ann(\chi)e$ for $r\in Re$.  Our next proposition describes $e\wh{V}$.

\begin{Prop}\label{p:ehatR}
Let $R$ be a finite commutative ring, $V$ a finite $R$-module  and $e\in E(R)$.  Then the following are equivalent for $\chi\in \wh{V}$.
\begin{enumerate}
\item $\chi\in e\wh{V}$;
\item $\chi=\gamma\circ \pi_e$ with $\gamma\in \wh{eV}$;
\item $(1-e)V\subseteq \ker \chi$.
\end{enumerate}
\end{Prop}
\begin{proof}
If $\chi\in e\wh{V}$ and $v\in (1-e)V$, then \[\chi(v) = (e\chi)(v)=\chi(ev) = \chi(e(1-e)v)=\chi(0)=1\] and so $(1-e)V\subseteq\ker \chi$.  If $(1-e)V\subseteq \ker \chi$, then $\chi$ factors through $\pi_e$ as $(1-e)V=\ker \pi_e$.  If $\chi=\gamma\circ \pi_e$ with $\gamma\in \wh{eV}$, then $(e\chi)(v) = \chi(ev) = \gamma(\pi_e(ev)) = \gamma(eev)=\gamma(ev) = \gamma(\pi_e(v)) = \chi(v)$ and so $e\chi=\chi$.  Therefore, $\chi\in e\wh{V}$.
\end{proof}

It follows that we can identify $e\wh {V}$ with $\wh{eV}$ and that the action of $U(Re)$ on $e\wh{V}$ by automorphisms can be identified with the action of $U(Re)$ on $\wh{eV}$ by automorphisms.  In particular, the orbits of $U(Re)$ on $e\wh{V}$ and $U(R)$ on $e\wh{V}$ are the same because $\p_e\colon U(R)\to U(Re)$ given by $r\mapsto re$ is onto by Proposition~\ref{p:unit.onto}.

\subsection{The representation theory of the affine monoid}\label{ss:rep.aff}
 Using Theorem~\ref{t:CMP} we can completely describe the simple $\mathbb C\Aff(V)$-modules for a finite module $V$ over a finite commutative ring $R$. We take for a complete set of idempotent representatives of the regular $\mathscr J$-classes of $\Aff(V)$ the mappings $g_e(x)=ex$ with $e\in E(R)$ (this is justified by Proposition~\ref{p:Aff.mon.idem}).  By Proposition~\ref{p:idem.local} we can identify the maximal subgroup at $g_e(x)$ with $U(\Aff(eV))$.  Let $\mathrm{Sing}(\Aff(eV)) = \Aff(eV)\setminus U(\Aff(eV))$ be the ideal of singular affine mappings.  Then $\mathbb CU(\Aff(eV))\cong \mathbb C\Aff(eV)/\mathbb C\mathrm{Sing}(\Aff(eV))$ and hence any simple $\mathbb CU(\Aff(eV))$-module $W$ can be viewed as a simple $\mathbb C\Aff(eV)$-module via inflation. The surjective homomorphism $\Phi_e\colon \Aff(V)\to \Aff(eV)$ from Proposition~\ref{p:extend.proj} extends to a surjective homomorphism of $\mathbb C$-algebras \[\Phi_e\colon \mathbb C\Aff(V)\to\mathbb C\Aff(eV)\] and hence $W$ becomes a simple $\mathbb C\Aff(V)$-module via inflation along $\Phi_e$.  Concretely, if $w\in W$ and $ax+b\in \Aff(V)$, then the action of $ax+b$ on $w$ is given by
\begin{equation}\label{eq:module.action}
(ax+b)w = \begin{cases}(aex+eb)w, & \text{if}\ Ra\supseteq Re\\ 0, & \text{else}\end{cases}
\end{equation}
 in light of Proposition~\ref{p:unit.onto}.  Clearly, $g_e(x)$ is an apex for this $\mathbb C\Aff(V)$-module structure on $W$ and $g_e(x)W=W$ as a $\mathbb CU(\Aff(eV))$-module.  Therefore, all the simple $\mathbb C\Aff(V)$-modules are obtained in this fashion by Theorem~\ref{t:CMP}; that is, $W^\sharp=W$ as a vector space with the module structure given by \eqref{eq:module.action}.

It thus remains to describe the representation theory of $\mathbb CU(\Aff(M))$ for a finite module $M$ over a finite commutative ring $S$.  The case of interest for us will be rings of the form $S=Re$ and modules of the form $M=eV$ with $e\in E(R)$.
This is a very special case of the representation theory of semidirect products of the form $A\rtimes G$ with $A$ an abelian group and $G$ an arbitrary group that can be found in standard texts on group representation theory (cf.~\cite{serrerep}).  Here we use the semidirect product decomposition $U(\Aff(M)) = M\rtimes U(S)$.

The group $U(S)$ acts on $\wh{M}$ via $(s\chi)(v) = \chi(sv)$ for $s\in U(S)$ and $v\in M$.  If $\chi\in \wh{M}$, let $\Stab_{U(S)}(\chi)$ be the stabilizer of $\chi$ in $U(S)$. If $\rho\in \wh{\Stab_{U(S)}(\chi)}$, then we can define a degree one character $\chi\otimes \rho\colon M\rtimes \Stab_{U(S)}(\chi)\to U(\mathbb C)$ by $(\chi\otimes \rho)(ax+b) = \chi(b)\rho(a)$ for $a\in \Stab_{U(S)}(\chi)$ and $b\in M$.

\begin{Thm}\label{t:sdir.irred}
Let $M$ be a finite $S$-module with $S$ a finite commutative ring.  Let $\mathcal O_1,\ldots, \mathcal O_m$ be the orbits of $U(S)$ on $\wh{M}$ and fix $\chi_i\in \mathcal O_i$.  Then a complete set of representatives of the isomorphism classes of simple $\mathbb CU(\Aff(M))$-modules is given by the modules \[W_{(\mathcal O_i,\rho)} = \Ind_{M\rtimes \Stab_{U(S)}(\chi_i)}^{U(\Aff(M))} \chi_i\otimes \rho\] with $\rho\in \wh{\Stab_{U(S)}(\chi_i)}$.
\end{Thm}
\begin{proof}
This is the specialization to $U(\Aff(M))=M\rtimes U(S)$ of the general theory of irreducible representations of semidirect products $A\rtimes G$ with $A$ abelian described in~\cite[Proposition~25]{serrerep}.
\end{proof}

Let $f\in E(S)$.  Then $U(\Aff(fM))$ acts transitively on $fM$ by permutations via the natural action $g(v)=av+b$ for $g(x) =ax+b\in U(\Aff(fM))$ and $v\in fM$.  Hence $U(\Aff(M))$ acts transitively on $fM$ by permutations via inflation along the surjective homomorphism $\Phi_f\colon U(\Aff(M))\to U(\Aff(fM))$ given by $\Phi_f(ax+b) = afx+fb$, cf.~Proposition~\ref{p:extend.proj}.  Notice that $ax+b\in U(\Aff(M))$ stabilizes $0$ under this action if and only if $fb=0$, which occurs if and only if $b\in (1-f)M$.  Hence $\Stab_{U(\Aff(M))}(0) = (1-f)M\rtimes U(S)$.   Therefore, the corresponding permutation module $\mathbb CfM$ for $\mathbb CU(\Aff(M))$ is the induced module \[\Ind_{(1-f)M\rtimes U(S)}^{U(\Aff(M))} \mathbf 1_{(1-f)M\rtimes U(S)}\] where we recall that $\mathbf 1_G$ denotes the trivial representation of a group $G$.  We record this as the first item of the following proposition.

\begin{Prop}\label{p:Sf.as.a.mod}
Let $S$ be a finite commutative ring, $M$ a finite $S$-module and $f\in E(S)$.  Let $\mathcal O_1,\ldots, \mathcal O_t$ be the orbits of $U(S)$ on $f\wh{M}$ (which is a  $U(\Aff(M))$-invariant subgroup of $\wh{M}$) and let $\chi_i\in \mathcal O_i$.
\begin{enumerate}
  \item The module $\mathbb CfM=\Ind_{(1-f)M\rtimes U(S)}^{U(\Aff(M))} \mathbf 1_{(1-f)M\rtimes U(S)}$.
  \item The decomposition of $\mathbb CfM$ into simple $\mathbb CU(\Aff(M))$-modules is giv\-en by
  \[
  \mathbb CfM = \bigoplus_{i=1}^t W_{(\mathcal O_i,\mathbf 1_{\Stab_
{U(S)}(\chi_i)})}
\]
(retaining the notation of Theorem~\ref{t:sdir.irred}).
\end{enumerate}
\end{Prop}
\begin{proof}
The first item was proved in the discussion immediately preceding the statement of the proposition.  To prove the second item, we apply Frobenius reciprocity and the Mackey decomposition.  Let $\chi\in \wh{M}$, $\rho\in \wh{\Stab_{U(S)}(\chi)}$ and let $\mathcal O$ be the orbit $\chi$ under $U(S)$. To decongest notation, we shall identify $M$ with the subgroup of translations and $U(S)$ with the subgroup of dilations and use the notation of internal semidirect products. Then by Frobenius reciprocity we have
\begin{equation}
\label{eq:decomp.perm.modu}
\begin{split}
& \left[\Ind_{(1-f)M\cdot U(S)}^{U(\Aff(M))} \mathbf 1_{(1-f)M\cdot U(S)}:W_{(\mathcal O,\rho)} \right] \\
& = \left[\Res^{U(\Aff(M))}_{(1-f)M\cdot U(S)} W_{(\mathcal O,\rho)}:\mathbf 1_{(1-f)M\cdot U(S)}\right].
\end{split}
\end{equation}

To compute the right hand side of \eqref{eq:decomp.perm.modu}, we apply the Mackey decomposition to \[W_{(\mathcal O,\rho)}=\Ind_{M\cdot \Stab_{U(S)}(\chi)}^{U(\Aff(M))} \chi\otimes \rho.\]  Let $h(x)\in U(\Aff(M))$.  Then, since $M$ is a normal subgroup of $U(\Aff(M))$, we have
\begin{align*}
& (1-f)M\cdot U(S)h(x)M\cdot \Stab_{U(S)}(\chi) \\
& = (1-f)M\cdot U(S)\cdot Mh(x)\Stab_{U(S)}(\chi)=U(\Aff(M)),
\end{align*}
and so there is only one double coset.  Therefore, using the Mackey decomposition and  that
\[
((1-f)M\cdot U(S))\cap (M\cdot \Stab_{U(S)}(\chi))=(1-f)M\cdot \Stab_{U(S)}(\chi)
\]
yields
\begin{equation}
\label{eq:decomp.perm.modu.two}
\begin{split}
& \Res^{U(\Aff(M))}_{(1-f)M\cdot U(S)} \Ind_{M\cdot \Stab_{U(S)}(\chi)}^{U(\Aff(M))} \chi\otimes \rho \\
&= \Ind_{(1-f)M\cdot \Stab_{U(S)}(\chi)}^{(1-f)M\cdot U(S)} \Res^{M\cdot \Stab_{U(S)}(\chi)}_{(1-f)M\cdot \Stab_{U(S)}(\chi)}\chi\otimes \rho.
\end{split}
\end{equation}
Another application of Frobenius reciprocity to \eqref{eq:decomp.perm.modu.two} shows that the right hand side of \eqref{eq:decomp.perm.modu} is equal to
\[
\left[\Res^{(1-f)M\cdot U(S)}_{(1-f)M\cdot \Stab_{U(S)}(\chi)} \mathbf 1_{(1-f)M\cdot U(S)}:\Res^{M\cdot \Stab_{U(S)}(\chi)}_{(1-f)M\cdot \Stab_{U(S)}(\chi)}\chi\otimes \rho\right],
\]
which is $1$ if $(1-f)M\subseteq \ker \chi$ and $\rho=\mathbf 1_{\Stab_{U(S)}}(\chi)$ and $0$, otherwise.  By Proposition~\ref{p:ehatR}, $(1-f)M\subseteq \ker \chi$ if and only if $\chi\in f\wh{M}$. This completes the proof.
\end{proof}

\begin{Rmk}\label{r:further.down}
We remark that if $e\in E(S)$ and $f\in E(Se)$, then the $\mathbb C\Aff(M)$-module structure on $\mathbb CfM$ is the inflation along $\Phi_e\colon \Aff(M)\to \Aff(eM)$ of the $\mathbb C\Aff(eM)$-module $\mathbb CfeM=\mathbb CfM$.  To relate the decomposition in \eqref{eq:decomp.perm.modu} over $\Aff(M)$ with the corresponding decomposition over $\Aff(eM)$, we should identify $f\wh M$ with $f\wh{eM}$ (both of which are isomorphic to $\wh{fM}$) and use Proposition~\ref{p:inflate.ind}.
\end{Rmk}

It will also be convenient to decompose a module of the form $W_{(\mathcal O,\mathbf 1_{\Stab_{U(S)}(\chi)})}$ with $\chi\in \mathcal O$ over $U(S)$.

\begin{Prop}\label{p:decomp.over.mult}
Let $\mathcal O$ be an orbit of $U(S)$ on $\wh M$ and $\chi\in \mathcal O$.  Then
\[\Res^{U(\Aff(M))}_{U(S)} W_{(\mathcal O,\mathbf 1_{\Stab_{U(S)}(\chi)})} = \bigoplus_{\substack{\rho\in \wh{U(S)},\\ \Stab_{U(S)}(\chi)\subseteq \ker \rho}} \rho.\]
\end{Prop}
\begin{proof}
We again identify $M$ with the subgroup of translations and $U(S)$ with the subgroup of dilations and use internal semidirect product notation.  If $h(x)\in U(\Aff(M))$, then the double coset \[U(S)h(x)M\cdot \Stab_{U(S)}(\chi) = U(S)\cdot Mh(x)\cdot \Stab_{U(S)}(\chi) = U(\Aff(M))\] by normality of $M$ and so there is only one double coset.  The Mackey decomposition and the equality $U(S)\cap M\cdot \Stab_{U(S)}(\chi) =\Stab_{U(S)}(\chi)$ then yield
\begin{equation}\label{eq:res.mult}
\Res^{U(\Aff(M))}_{U(S)}\Ind_{M\cdot \Stab_{U(S)}(\chi)}^{U(\Aff(M))} \chi\otimes \mathbf 1_{\Stab_{U(S)}}= \Ind_{\Stab_{U(S)}(\chi)}^{U(S)} \mathbf 1_{\Stab_{U(S)}(\chi)}.
\end{equation}
Applying Frobenius reciprocity to \eqref{eq:res.mult} shows that the multiplicity of $\rho\in \wh{U(S)}$ as a summand in the right hand side of \eqref{eq:res.mult} is $1$ if $\rho|_{\Stab_{U(S)}(\chi)}=\mathbf 1_{\Stab_{U(S)}(\chi)}$ and $0$, otherwise.  This completes the proof.
\end{proof}

Let us return now to our original finite commutative ring $R$ and a finite $R$-module $V$.  We wish to find the composition factors of $\mathbb CV$ as a $\mathbb C\Aff(V)$-module.  The module structure on $\mathbb CV$ is just the linear extension of the natural action where $h(x)=ax+b$ acts on $v\in V$ by $h(v)=av+b$.  The character $\theta$ of this module is given by \[\theta(ax+b) = \left|\{v\in V\mid av+b=v\}\right|.\] Let $\mu$ denote the M\"obius function of the lattice $\Lambda(R)$ of idempotent generated principal ideals of $R$ and $\zeta$ its zeta function. Note that if $e,f\in E(R)$, then $Re=Rf$ if and only if $e=f$ and $Rf\subseteq Re$ if and only if $f=ef=fe$.  We shall need the following observation.  If $\chi \in \wh{V}$ and if $e\chi=\chi=f\chi$, then $ef\chi=\chi$. Hence, there exists $e_{\chi}\in E(R)$ such that $e_{\chi}\chi=\chi$ and, for all $f\in E(R)$, $f\chi=f$  if and only if $Re_{\chi}\subseteq Rf$.  Also note that if $\chi$ and $\chi'$ are in the same $U(R)$ orbit, then $e_{\chi}=e_{\chi'}$ because $f\chi=\chi$ if and only if $f\chi'=\chi'$.  Thus we put $e_{\mathcal O}=e_{\chi}$ for any $\chi$ in the orbit $\mathcal O$ of $U(R)$ on $\wh{V}$.

\begin{Thm}\label{t:comp.factors}
Let $R$ be a finite commutative ring, $V$ a finite $R$-module and $e\in E(R)$.  Let $\mathcal O_1,\ldots, \mathcal O_s$ be the orbits of $U(Re)$ on $e\wh{V}$, which we may identify with $\wh{eV}$. Let $\chi_i\in \mathcal O_i$.  Then the composition factors of $\mathbb CV$ with apex $e$ are exactly those $W_{(\mathcal O_i,\mathbf 1_{\Stab_{U(Re)}(\chi_i)})}^\sharp$ with \[\mathcal O_i\subseteq e\wh{V}\setminus \bigcup_{\substack{Rf\subsetneq Re,\\ f\in E(R)}}f\wh{V},\] that is, with $e_{\mathcal O_i}=e$ and they each  appear with multiplicity one.
\end{Thm}
\begin{proof}
Let $\gamma$ be the character of $W_{(\mathcal O_i,\rho)}$. We again put $g_f(x)=fx$ for $f\in E(R)$. Note that if $h(x)=ax+b$, then $g_fhg_f(x)=g_fh(x) = afx+fb$.
Since $E(\Aff(V))$ is a submonoid by Proposition~\ref{p:Aff.mon.idem}, Theorem~\ref{t:mult}  yields that $[\mathbb CV:W_{(\mathcal O_i,\rho)}^\sharp]$, with $\rho\in \wh{\Stab_{U(Re)}(\chi_i)}$, is given by
\begin{equation}\label{eq:mobius}
\frac{1}{|\Aff(eV)|}\sum_{h(x)\in U(\Aff(eV))}\gamma(h(x))\sum_{Rf\subseteq Re}\overline{\theta(g_fh(x)g_f)}\mu(Rf,Re).
\end{equation}
But note that $\theta(g_fh(x)g_f) = \theta(g_fh(x)) = \theta(\Phi_f(h(x)))$, for $Rf\subseteq Re$, where \[\Phi_f\colon U(\Aff(eV))\to U(\Aff(fV))\] is the surjective homomorphism $\Phi_f(ax+b)=afx+fb$, cf.~Proposition~\ref{p:extend.proj}.  Set $S=Re$ and put $M=eV$, which is an $S$-module.  Note that $afv+fb=v$ implies that $v\in fV=fM$ and so $\theta\circ \Phi_f$ is the character of the $\mathbb CS$-module $\mathbb CfM$ of Proposition~\ref{p:Sf.as.a.mod}.  Thus, applying the first orthogonality relations, \eqref{eq:mobius} becomes
\begin{equation}\label{eq:mobius2}
\begin{split}
\sum_{Rf\subseteq Re}\mu(Rf,Re)\frac{1}{|\Aff(eV)|}\sum_{h(x)\in U(\Aff(eV))}\gamma(h(x))\overline{\theta(\Phi_f(h(x))}\\=\sum_{Rf\subseteq Re}[\mathbb CfM:W_{(\mathcal O_i,\rho)}]\mu(Rf,Re).
\end{split}
\end{equation}
But Proposition~\ref{p:Sf.as.a.mod} shows that this multiplicity is zero unless $\mathcal O_i\subseteq f\wh M=f\wh V$ (under the identification of both with $\wh{fV}$, cf.~Remark~\ref{r:further.down}) and $\rho=\mathbf 1_{\Stab_{U(Re)}(\chi_i)}$, in which case it is one.  Therefore, the right hand side of \eqref{eq:mobius2} equals zero unless $\rho=\mathbf 1_{\Stab_{U(Re)}(\chi_i)}$, in which case it is
\begin{equation}\label{eq:mobius3}
\sum_{Re_{\mathcal O_i}\subseteq Rf\subseteq Re}\zeta(Re_{\mathcal O_i},Rf)\mu(Rf,Re).
\end{equation}
But the quantity in \eqref{eq:mobius3} is zero unless $e_{\mathcal O_i}=e$, in which case it is one.  This completes the proof of the theorem.
\end{proof}

\section{Eigenvalues}
\label{sec:eigen}
Fix a finite commutative ring $R$ and a finite $R$-module $V$. Our goal is to prove Theorem~\ref{t:main2}.  We shall, in fact, prove a more general  result about the eigenvalues of certain elements of $\mathbb C\Aff(V)$ acting on $\mathbb CV$.  Let us begin with a description of how an operator supported on translations and constant on associates acts under an irreducible representation.

\begin{Prop}\label{p:add.part}
Let $P\in \mathbb CV$ be constant on associates.  View $P$ as an element of $\mathbb C\Aff(V)$ supported on translations. Let $W_{(\mathcal O,\rho)}^\sharp$ be a simple $\mathbb C\Aff(V)$-module with apex $e$, whence $\mathcal O=U(Re)\chi$ is an orbit of $U(Re)$ on $\wh{eV}$ (which we identify with $e\wh V$ and hence we identify $\mathcal O$ with $U(R)\chi$) and $\rho$ is a character of $\Stab_{U(Re)}(\chi)$.
Then $P$ acts on $W_{(\mathcal O,\rho)}^\sharp$ via scalar multiplication by
\[\wh P(\chi) = \sum_{b\in V} P(b)\chi(b).\]
\end{Prop}
\begin{proof}
Note that $P\in \mathbb CU(\Aff(V))$.  Moreover, notice that two translations $g(x)=x+b$ and $h(x)=x+c$ are conjugate in $U(\Aff(V))$ if and only if $b,c$ are associates.  Indeed, conjugating a translation by a translation does nothing. On the other hand, conjugating $g(x)=x+b$ by $h(x)=ux$ with $u\in U(R)$ yields $hgh^{-1}(x) = x+ub$.  Thus $P$ belongs to the center of $\mathbb CU(\Aff(V))$.  Therefore, $P$ acts via a scalar on any simple $\mathbb CU(\Aff(V))$-module by Schur's lemma.

 Let $\Phi_e\colon U(\Aff(V))\to U(\Aff(eV))$ be the canonical homomorphism.  We saw at the beginning of Subsection~\ref{ss:rep.aff}
  that the restriction of $W_{(\mathcal O,\rho)}^\sharp$ to $U(\Aff(V))$ is the inflation of $W_{(\mathcal O,\rho)}$ along $\Phi_e$.  As $\Phi_e$ is surjective by Proposition~\ref{p:extend.proj}, this is a simple $\mathbb CU(\Aff(V))$-module.  In fact, it is the simple module $W_{(\mathcal O,\rho\p_e)}$, where $\p_e\colon U(R)\to U(Re)$ is the projection and we view $\mathcal O$ as an orbit of $U(Re)$ on  $\wh{eV}$, by Proposition~\ref{p:inflate.ind}.  It remains to understand the restriction of $W_{(\mathcal O,\rho\p_e)}$ to $V$ (viewed as the group of translations in $U(\Aff(V))$).

As $U(R)$ is a set of coset representatives for the normal subgroup $V$ of $U(\Aff(V))$, conjugation of a translation $h(x)=x+b$ by a dilation $g(x)=ux$ corresponds to multiplying $b$ by $u$ and $(\chi\otimes \rho\p_e)|_V=\chi$, the Mackey decomposition yields that as a $\mathbb CV$-module $W_{(\mathcal O,\rho\p_e)}$ is the direct sum of the characters in the orbit $\mathcal O$ of $\chi$.   Each of these characters give the same Fourier transformation of $P$ because $P$ is constant on associates.  Thus $P$ acts on $W_{(\mathcal O,\rho\p_e)}$ via scalar multiplication by $\wh P(\chi)$.        This completes the proof.
\end{proof}

We remark that it is almost never the case that $P$ belongs to the center of $\mathbb C\Aff(V)$.  Indeed, if $z(x)\in \Aff(V)$ is the zero mapping, then $zP=z$.  But if $P$ is not a point mass at $0$, then $Pz=\sum_{b\in V}P(b)(0x+b)\neq z$.  Thus we are using in an essential way the observation, implicit in the above proof, that each irreducible representation of $\Aff(V)$ remains irreducible when restricted to the group of units $U(\Aff(V))$ in order to apply Schur's lemma and deduce that $P$ acts as a scalar matrix under irreducible representations of $\Aff(V)$.

Let $p(x,y)\in \mathbb C[x,y]$ be a polynomial and let $P,Q\in \mathbb C\Aff(V)$ with $P$ supported on translations and $Q$ supported on dilations, i.e., $P$ is supported on $V$ and $Q$ is supported on $M(R)$ under the semidirect product decomposition $\Aff(V)= V\rtimes M(R)$.  We further assume that $P$ is constant on associates.  Then we compute the eigenvalues of $A=p(P,Q)$ on the module $\mathbb CV$.  More precisely, we prove the following theorem.

\begin{Thm}\label{t:eigenvalues}
Let $p(x,y)\in \mathbb C[x,y]$ be a polynomial and let $P,Q\in \mathbb C\Aff(V)$ with $P$ supported on translations and $Q$ supported on dilations.  We further assume that $P$ is constant on associates.  Put $A=p(P,Q)$.
Then the eigenvalues for $A$ on $\mathbb CV$ are indexed by triples $(e,\mathcal O,\rho)$ where:
\begin{enumerate}
\item $e\in E(R)$;
\item $\mathcal O=U(R)\chi$ is an orbit of $U(R)$ on \[e\wh{V}\setminus \bigcup_{\substack{Rf\subsetneq Re,\\ f\in E(R)}}f\wh{V},\] that is, $e_{\mathcal O}=e$;
\item and $\rho\in \wh{U(Re)}$ such that $\Stab_{U(Re)}(\chi)\subseteq \ker \rho$.
\end{enumerate}
The corresponding eigenvalue is $p(\wh P(\chi),\wh Q(\rho))$, where
\begin{align*}
\wh P(\chi) &= \sum_{b\in V}P(b)\chi(b)\\
\wh Q(\rho) &= \sum_{Ra\supseteq Re} Q(a)\rho(ae),
\end{align*}
 and it occurs with multiplicity one.
\end{Thm}
\begin{proof}
Let $\mathbb CV=U_0\supsetneq U_1\supsetneq \cdots \supsetneq U_n=0$ be a composition series for $\mathbb CV$ as a $\mathbb C\Aff(V)$-module.  If we choose a basis $B=B_0\cup\cdots\cup B_{n-1}$ for $\mathbb CV$ such that $B_i\subseteq U_i$ projects to a basis for $U_i/U_{i+1}$ for $0\leq i\leq n-1$, then the corresponding matrix representation of $\Aff(V)$ has a block upper triangular form with diagonal blocks $\rho_i$ corresponding to the matrix representation afforded by $U_i/U_{i+1}$ with respect to the basis that is the projection of $B_i$ into $U_i/U_{i+1}$.  It follows that the set of eigenvalues of $A$ on $\mathbb CV$ with multiplicities is the union with multiplicities of the eigenvalues of $\rho_i(A)$ for $1\leq i\leq n-1$, i.e., for the action of $A$ on the composition factors of $\mathbb CV$.   The composition factors of $\mathbb CV$ are described in Theorem~\ref{t:comp.factors}.  So it suffices to show that if $\chi\in \wh{V}$ and $e=e_{\chi}$, then the eigenvalues of $A$ on $W_{(\mathcal O,\mathbf 1_{\Stab_{U(Re)}(\chi)})}^\sharp$, where $\mathcal O$ is the orbit of $\chi$ under $U(R)$, are of the form $p(\wh P(\chi),\wh Q(\rho))$ where $\rho\in \wh{U(Re)}$ with $\Stab_{U(Re)}(\chi)\subseteq \ker \rho$.

By Proposition~\ref{p:add.part},  $P$ acts on $W_{(\mathcal O,\mathbf 1_{\Stab_{U(Re)}(\chi)})}^\sharp$ as scalar multiplication by $\wh P(\chi)$.  It therefore suffices to show that $Q$ is diagonalizable with eigenvalues $\wh Q(\rho)$ where $\rho\in \wh{U(Re)}$ with $\Stab_{U(Re)}(\chi)\subseteq \ker \rho$.  But $Q$ is supported on dilations and the dilation $h(x)=ax$ acts on $W_{(\mathcal O,\mathbf 1_{\Stab_{U(Re)}(\chi)})}^\sharp$ as the element $x\mapsto aex$ of $U(Re)$ if $Ra\supseteq Re$ and as $0$, else.   Proposition~\ref{p:decomp.over.mult} shows that $W_{(\mathcal O,\mathbf 1_{\Stab_{U(Re)}(\chi)})}$ restricts to $U(Re)$ as a direct sum of precisely the linear characters $\rho$ where $\rho\in \wh{U(Re)}$ with $\Stab_{U(Re)}(\chi)\subseteq \ker \rho$ and hence on the corresponding summand $Q$ acts as $\wh Q(\rho)$.   This completes the proof.
\end{proof}

The special case where $P,Q$ are probabilities and $p(x,y) = \alpha x+(1-\alpha)y$ with $0\leq \alpha\leq 1$ for the coin-toss walk and $p(x,y) = xy$ for the  affine walk (cf.~ Subsection~\ref{ss.monoid.random}) yields the following formulation of Theorem~\ref{t:main2}.

\begin{Thm}\label{t:main}
Let $R$ be a finite commutative ring, $V$ a finite $R$-module, $P$ a probability on $V$ that is constant on associates and $Q$ a probability on $R$.
Then the eigenvalues for the transition matrices of both the coin-toss walk and the  affine walk on $V$ are indexed by triples $(e,\mathcal O,\rho)$ where:
\begin{enumerate}
\item $e\in E(R)$;
\item $\mathcal O=U(R)\chi$ is an orbit of $U(R)$ on \[e\wh{V}\setminus \bigcup_{\substack{Rf\subsetneq Re,\\ f\in E(R)}}f\wh{V};\]
\item and $\rho\in \wh{U(Re)}$ such that $\Stab_{U(Re)}(\chi)\subseteq \ker \rho$.
\end{enumerate}
The corresponding eigenvalue for the coin-toss walk is $\alpha\wh P(\chi)+(1-\alpha)\wh Q(\rho)$ and for the affine walk is $\wh P(\chi)\wh Q(\rho)$ (where $\chi\in \mathcal O$) where
\begin{align*}
\wh P(\chi) &= \sum_{b\in V}P(b)\chi(b)\\
\wh Q(\rho) &= \sum_{Ra\supseteq Re} Q(a)\rho(ae).
\end{align*}
 This eigenvalue occurs with multiplicity one.
\end{Thm}

To see that Theorem~\ref{t:main2} is a reformulation of Theorem~\ref{t:main}, we need some further preliminaries. Let $W$ be a finite $R$-module. If $v\in W$, then \[\mathrm{ann}(v)=\{a\in R\mid av=0\}\] is the \emph{annihilator} of $v$; it is an ideal of $R$.  Note that the annihilators of $v$ and the cyclic submodule $Rv$ are the same and $Rv\cong R/\ann(v)$ as an $R$-module.  In particular, $Rv\cong Rw$ if and only if $\ann(v)=\ann(w)$.

\begin{Prop}\label{p:ann.ideal}
Let $v\in V$ and $e\in E(R)$.  Then $ev=v$ if and only if $1-e\in E(\ann(v))$.	
\end{Prop}
\begin{proof}
This is obvious.	
\end{proof}

Recall that $E(R)$ is a Boolean algebra with respect to the ordering $e\leq f$ if $ef=e$.  The  operations are given by $e\wedge f=ef$, $e\vee f=e+f-ef$ and $\neg e=1-e$.  The mapping $e\mapsto Re$ is an isomorphism of Boolean algebras.

As $ev=v$ and $fv=v$ implies $efv=v$, it follows that there is a minimal idempotent $e_v$ with $e_vv=v$.  As $e\mapsto 1-e$ is an order reversing involution on $E(R)$, it follows from Proposition~\ref{p:ann.ideal} that $e=e_v$ if and only if $1-e$ is the maximal idempotent in $\mathrm{ann}(v)$ (which is a join subsemilattice as $0\in \mathrm{ann}(v)$ and $e\vee f= e+f-ef$).  In particular, $\mathrm{ann}(v)=\mathrm{ann}(w)$ implies $e_v=e_w$.  Also note that $Rv=Rw$ implies that $\mathrm{ann}(v)=\mathrm{ann}(w)$, whence $e_v=e_w$.

\begin{Prop}\label{p:ann.stuff}
Let $v\in W$.
\begin{enumerate}
\item The natural mapping $\pi_v\colon U(R)\to U(R/\ann(v))$ is surjective and $\pi_v(r)=\pi_v(re_v)$ for all $r\in R$.
\item Let $\rho\in \wh{U(R)}$.  Then $(1+\ann(v))\cap U(R)\subseteq \ker \rho$ if and only if $\rho$ factors through $\pi_v$.
\item If $a\in Re_v$, then $a\in U(Re_v)$ if and only if $a+\ann(v)\in U(R/\ann(v))$.
\end{enumerate}	
\end{Prop}
\begin{proof}
First note that $r-re_v=r(1-e_v)\in \ann(v)$ for all $r\in R$, whence $\pi_v(r)=\pi_v(re_v)$, and so $Re_v\to R/\ann(a)$ is a surjective homomorphism of unital rings. It is enough to show that the natural mapping $U(Re_v)\to U(R/\ann(v))$ is surjective since we can then apply Proposition~\ref{p:unit.onto}.

Suppose that $a+\ann(v)$ is a unit with inverse $b+\ann(v)$ with $a,b\in Re_v$.  Then $ab+\ann(v)=1+\ann(v)=e_v+\ann(v)$ and so $ab=e_v+x$ with $x\in \ann(v)$.  Then $abv=e_vv+xv=v$ and hence $(ab)^nv=v$ for all $n>0$.  As $M(R)$ is a  finite monoid, we have that $(ab)^k\in E(Re_v)$ for some $k>0$.  Minimality of $e_v$ then implies $(ab)^k=e_v$ and so $a\in U(Re_v)$.  This completes the proof of the first item and also the third.

The second item follows from the first and the observation that $\ker \pi_v=(1+\ann(v))\cap U(R)$.
\end{proof}

\begin{proof}[Proof of Theorem~\ref{t:main2}]
By Proposition~\ref{p:cyclic.sub} there is a bijection between orbits of $U(R)$ on $\wh{V}$ and cyclic submodules if $\wh{V}$.  Moreover, $e$ is the minimal idempotent stabilizing $\chi\in \wh V$ if and only if $1-e$ is the maximal idempotent in $\ann(\chi)$ by Proposition~\ref{p:ann.ideal}.  If $e$ is the minimal idempotent stabilizing $\chi$, then $\Stab_{U(Re)}(\chi) = e+\ann(\chi)e$.  By Proposition~\ref{p:ann.stuff}, we have that if $\rho\in \wh{U(Re)}$, then $\ker \rho$ contains $\Stab_{U(Re)}(\chi)$ if and only if $\rho$ factors through $Re\to R/\ann(\chi)$. Thus the triples $(e,\mathcal O,\rho)$ from Theorem~\ref{t:main} correspond bijectively to the pairs $(W,\rho)$ of Theorem~\ref{t:main}.    Moreover, by Proposition~\ref{p:unit.onto} and Proposition~\ref{p:ann.stuff}, we have that $Ra\supseteq Re$ if and only if $ae\in U(Re)$, if and only if $ae+\ann(\chi)=a+\ann(\chi)\in U(R/\ann(\chi))$, if and only if $a\in U(R)+\ann(\chi)$.  It follows that the definitions of $\wh Q(\rho)$ in Theorem~\ref{t:main2} and in Theorem~\ref{t:main} agree (using Proposition~\ref{p:ann.stuff}).
\end{proof}

We now aim to recover the results of the first author and Singla~\cite{AyyerSingla}  for the case when $V=R$  (and hence cyclic submodules are principal ideals) and $P$ is uniform. In this case, one has that $\wh P(\chi)=0$ for any non-trivial character $\chi$ on the additive group of $R$ by the orthogonality relations for characters.  Hence many of the eigenvalues in Theorem~\ref{t:main2} will be the same.  We shall first give a description of the eigenvalues that follows directly from Theorem~\ref{t:main2}.  We shall then reformulate the result to make it apparent that it agrees with the results of~\cite{AyyerSingla} for $V=R$.

\begin{Thm}\label{t:uniform.version.one}
Let $R$ be a finite commutative ring and $V$ a finite $R$-module.
Let $Q$ be a probability on $R$.  Then the eigenvalues for transition matrix of the coin-toss random walk on $V$ with respect to $P$ the uniform distribution and $Q$, with heads probability $\alpha$, are indexed by pairs $([W],\rho)$ where $[W]$ is the isomorphism class of a cyclic submodule of $W$ of $\wh{V}$ and $\rho$ is a character of $U(R)$ factoring through $\pi_W\colon U(R)\to U(R/\ann(W))$.  The corresponding eigenvalue is given by
\[\lambda_{([W],\rho)} = \begin{cases} 1, & \text{if}\ W=0\\ (1-\alpha)\sum_{a\in U(R)+\ann(W)}Q(a)\rho(u_W(a)), & \text{else}\end{cases}
\] where $u_W(a)\in U(R)$ with $u_W(a)+\ann(W)=a+\ann(W)$; it has multiplicity the number of cyclic submodules of $\wh{V}$ isomorphic to $W$.
\end{Thm}
\begin{proof}
This follows from Theorem~\ref{t:main2}, the observation that \[\wh P(\chi)=\frac{1}{|V|}\sum_{v\in V}\chi(v)=\langle  1_V,\chi\rangle=\begin{cases}1, & \text{if}\ \chi=\mathbf 1_V\\ 0, &\text{else}\end{cases}\] by the orthogonality relations and from Proposition~\ref{p:ann.stuff}.
\end{proof}

Our next goal is to show that we can work with cyclic submodules of $V$ instead of $\wh{V}$.  This is necessary to recover the result as formulated in~\cite{AyyerSingla}.  We do this by showing that the coin-toss walk with $P$  uniform for $V$ and $\wh{V}$ have the same eigenvalues.

The following result can also be proved via elementary linear algebra (cf.~\cite{DZ}).

\begin{Prop}\label{p:coin.toss.uniform}
Consider the coin-toss random walk on $V$ where $P$ is taken to be the uniform distribution on $R$, $Q$ is any distribution on $R$ and $\alpha$ is the probability of heads.  Let $1=\lambda_1, \lambda_2, \dots, \lambda_k$
be the eigenvalues for the transition matrix of the random walk of the multiplicative monoid $M(R)$ on $V$ driven by $Q$ with multiplicities.
Then the eigenvalues for the transition matrix of the coin-toss walk are
\[
1 =\lambda_1, (1-\alpha)\lambda_2, \dots, (1-\alpha)\lambda_k
\]
with multiplicities.
\end{Prop}
\begin{proof}
This is immediate from Theorem~\ref{t:uniform.version.one} once we observe that the random walk on $V$ driven by $Q$ is the coin-toss random walk with $\alpha =0$.
\end{proof}

By Proposition~\ref{p:coin.toss.uniform}, to show that the coin-toss walk with $P$ uniform for $V$ and $\wh{V}$ have the same eigenvalues, it suffices to show that the random walks of the multiplicative monoid $M(R)$ on $V$ and $\wh{V}$ driven by $Q$ have the same eigenvalues.

\begin{Prop}\label{p:same.eigenvalues}
Let $V$ be a finite $R$-module and $Q$ a probability on $R$.  Then the transition matrix for the random walk of $M(R)$ on $\wh{V}$ driven by $Q$ is similar to the transpose of the transition matrix of the random walk of $M(R)$ on  $V$ driven by $Q$ and hence both transition matrices have the same eigenvalues.
\end{Prop}
\begin{proof}
The vector space dual of $\mathbb CV$, which is a $\mathbb CM(R)$-module,  can be identified with the space of functions $f\colon V\to \mathbb C$ and the dual basis to the basis $V$ of $\mathbb CV$ corresponds to the indicator functions $\delta_v$ of the singleton sets $\{v\}$ with $v\in V$.  The module action of $R$ on $\mathbb C^V$ is given by $(rf)(v)= f(rv)$.  So the matrix of $Q$ acting on $\mathbb C^V$ with respect to the basis of indicator functions is the transpose of the matrix of $Q$ acting of $\mathbb CV$ with respect to the basis $\Omega$ and hence is the transition matrix of the random walk of $M(R)$ on $V$ driven by $Q$ (cf.~ Subsection~\ref{ss.monoid.random}).  But the characters of $V$ also form a basis for $\mathbb C^V$ and the matrix of $Q$ with respect to the basis of characters is the transpose of the transition matrix of the random walk of $M(R)$ on $\wh V$ driven by $Q$.  This completes the proof.
\end{proof}

Using Theorem~\ref{t:uniform.version.one}, Proposition~\ref{p:coin.toss.uniform}, Proposition~\ref{p:same.eigenvalues} and that $\wh {\wh{M}}$ is canonically isomorphic to $M$ as an $R$-module, we then obtain the following theorem.

\begin{Thm}\label{t:eigenvalues.two}
Let $R$ be a finite commutative ring and $V$ a finite $R$-module.
Let $Q$ be a probability on $R$.  Then the eigenvalues for transition matrix of the coin-toss random walk on $V$ with respect to $P$ the uniform distribution and $Q$, with heads probability $\alpha$, are indexed by pairs $([W],\rho)$ where $[W]$ is the isomorphism class of a cyclic submodule of $W$ of $V$ and $\rho$ is a character of $U(R)$ factoring through $\pi_W\colon U(R)\to U(R/\ann(W))$.  The corresponding eigenvalue is
\[\lambda_{([W],\rho)} = \begin{cases} 1, & \text{if}\ W=0\\ (1-\alpha)\sum_{a\in U(R)+\ann(W)}Q(a)\rho(u_W(a)), & \text{else}\end{cases}
\] where $u_W(a)\in U(R)$ with $u_W(a)+\ann(W)=a+\ann(W)$; it has multiplicity the number of cyclic submodules of $V$ isomorphic to $W$.
\end{Thm}

The special case in which $V=R$ (and hence cyclic modules are principal ideals) recovers~\cite[Theorem 2.3]{AyyerSingla}.

\def\malce{\mathbin{\hbox{$\bigcirc$\rlap{\kern-7.75pt\raise0,50pt\hbox{${\tt
  m}$}}}}}\def\cprime{$'$} \def\cprime{$'$} \def\cprime{$'$} \def\cprime{$'$}
  \def\cprime{$'$} \def\cprime{$'$} \def\cprime{$'$} \def\cprime{$'$}
  \def\cprime{$'$} \def\cprime{$'$}

\end{document}